\documentclass[12pt]{amsart}
\usepackage[latin9]{inputenc}
\usepackage{enumitem}
\usepackage{amstext}
\usepackage{amsthm}
\usepackage{amssymb}
\usepackage{graphicx}
\usepackage[margin=1in]{geometry}
\usepackage{autonum}
\usepackage[colorinlistoftodos]{todonotes}

\makeatletter

\newcommand{\noun}[1]{\textsc{#1}}

\newcommand{\ZZ}{\mathbb{Z}}

\numberwithin{equation}{section}
\numberwithin{figure}{section}
\theoremstyle{plain}
\newtheorem{thm}{\protect\theoremname}[section]
 \theoremstyle{plain}
 \newtheorem{lem}[thm]{\protect\lemmaname}
 \newlist{casenv}{enumerate}{4}
 \setlist[casenv]{leftmargin=*,align=left,widest={iiii}}
 \setlist[casenv,1]{label={{\itshape\ \casename} \arabic*.},ref=\arabic*}
 \setlist[casenv,2]{label={{\itshape\ \casename} \roman*.},ref=\roman*}
 \setlist[casenv,3]{label={{\itshape\ \casename\ \alph*.}},ref=\alph*}
 \setlist[casenv,4]{label={{\itshape\ \casename} \arabic*.},ref=\arabic*}

\makeatother

\usepackage[english]{babel}
 \providecommand{\lemmaname}{Lemma}
 \providecommand{\casename}{Case}
\providecommand{\theoremname}{Theorem}

\begin{document}

\title{Comply/Constrain Subtraction}

\author{Archishman Sravankumar}
\address{Euler Circle, Palo Alto, CA, 94306}
\email{archishman.sravankumar@gmail.com}

\date{\today}

\begin{abstract}
A comply/constrain game or a game with a Muller twist is a game where the next player is allowed to place constraints on
opponent's next move. We develop a closed form formula for
the Grundy value of the single-pile \noun{subtraction} game where the next
player may determine whether the previous player has to select a move
from the set of some first $k$ natural numbers or its complement. We 
also investigate the periodicity of Grundy values when the set of legal
moves is from a set of finite arithmetic sequences.
\end{abstract}

\maketitle

\section{Introduction}


In this paper, we consider a variant of the classic game of \textsc{subtraction}. In an instance of \textsc{subtraction}, we choose some subset $S\subseteq\ZZ^+$ (where $\ZZ^+$ for us does not include 0) for once and for all. A move in the game \textsc{subtraction}$_S(n)$ consists of removing some $s\in S$ stones from a pile of size $n$, leaving $n-s\ge 0$ stones. The loser is the first player who cannot make a move.
If $s\in S$ is such that $s>n$, then we ignore $s$ because we cannot take more than $n$ stones from a pile of $n$ stones. Therefore, the set of legal moves can be simplified to $S\cap[n]$ where $[n]=\{1,2,3,4,\ldots, n\}$. We can represent a \noun{subtraction} position by its options. The position $\textsc{subtraction}_S(n)$ can be defined in terms of its options as
\begin{equation}
\textsc{subtraction}_S(n)=\{(n-s,S):s\in S\cap[n]\}.
\end{equation}
The values of impartial games like \noun{subtraction} can be calculated
using the minimal excludant (mex) rule.

In this paper, we consider a modification of the \noun{subtraction} game, known as a \emph{Muller twist}, which we call \noun{comply/constrain subtraction}. A comply/constrain game (or a game with a Muller twist) is one where
the next player (the player whose turn it is to move) may add a condition
about the moves the opponent may make along with the physical
move itself. In the game of \noun{subtraction}, we introduce a Muller twist by allowing the next player to decide whether the opponent should make a move from $S$ or its complement $\overline{S}=\mathbb{Z}^+\setminus S$. We denote a game of \noun{comply/constrain subtraction} with $n$ stones and subtraction set $S$ as $(n,S)$.

Comply/constrain games were first studied in a paper by St\u{a}nic\u{a} and Smith \cite{SS02}. A comply/constrain game, or a game with a Muller twist, is a variant of a traditional game where a player's physical move of game pieces is followed by a constraint chosen from a well-defined set of constraints. Their paper considered \noun{odd-or-even nim}, \noun{tall-or-short Wyt queens}, and \noun{Fibonacci-or-not nim}, all of which are traditional combinatorial games with a Muller twist. Other variants of nim with a Muller twist include \noun{blocking nim} which was studied in \cite{FHR03}. A paper by Horrocks and Trenton \cite{key-2} considers the game of \noun{subtraction} with a Muller twist. They analyze the periodicity of Grundy values when the subtraction set is of the form $S = \{a : a \equiv b \mod c\}$ for some $b$ and $c$. 

Another related family of games that has recently attracted interest is the family of \emph{push-the-button} games; see~\cite{DHLP18}. In push-the-button games, there are two rulesets, say $A$ and $B$, played on the same heap set. The game initially starts with players moving according to the $A$-ruleset, but at any point, a player may ``push the button,'' and all subsequent moves are played according to the $B$-ruleset. Games with Muller twists are similar, except that the button can be pressed many times, thus switching back and forth between the rulesets.

This paper also fits in with recent interest in finding Grundy values of heap games. See for instance~\cite{LR16} for recent work on Grundy values of the game of \textsc{Fibonacci nim}, first studied in~\cite{Whinihan63}.

The options of the \noun{comply/constrain subtraction} $(n,S)$ are 
\begin{equation}
(n,S)=\{(n-s,S):s\in S\cap[n]\}\cup\{(n-s,\overline{S}):s\in S\cap[n]\}.
\end{equation}
We will use $\mathcal{G}(n,S)$
for the Grundy value of the game $(n,S)$. That is, if $(n,S)=*x$,
then $\mathcal{G}(n,S)=x$.

In \S \ref{sec:consec}, we develop a closed form expression for the Grundy values of
games where $S=[k]$ or $S=\overline{[k]}$. 
In \S \ref{sec:arith} we analyze the periodicity of Grundy values when $S=\{b+ic:\ 0\leq i\leq i_{\max},\frac{c+2}{2}\leq b<c\}$. These are similar to the games studied in \cite{key-2} except that we now cut off the arithmetic progression at some point. 

\section*{Acknowledgments}

I would like to thank Simon Rubinstein-Salzedo for helpful discussions.

\section{Consecutive integers} \label{sec:consec}
Here we investigate the Grundy values of games $(n,S)$ where $S = [k]:=\{1,2,\ldots,k\}$. While these sets are also arithmetic progressions, they fail other parts of the hypothesis required for the analysis in \S \ref{sec:arith}. Therefore, we investigate them separately.

\begin{lem}\label{thm:1}
For $n>k$, we have $\mathcal{G}(n+1,\overline{[k]})\geq \mathcal{G}(n,\overline{[k]})$.
\end{lem}

\begin{proof} 
From any $(n,\overline{[k]})$ a player can move to any or $(m,[k])$ or $(m,\overline{[k]})$ where $m<n-k$ since $\overline{[k]}$ contains all integers greater than or equal to $k + 1$. From $(n+1,\overline{[k]})$ a player can move to any $(m,[k])$ or $(m,\overline{[k]})$ where $m<n+1-k$. Therefore, all the options of $(n,\overline{[k]})$
are also options of $(n+1,\overline{[k]})$. Therefore, by the mex rule,
$\mathcal{G}(n+1,\overline{[k]})\geq\mathcal{G}(n,\overline{[k]})$. 
\end{proof}
\begin{thm}
With notation as in the introduction, we have 
\begin{equation}\label{eq:1a}
\mathcal{G}(n,[k])=\begin{cases}
n & 0\leq n\leq2k\\
n+1\mod(k+1) & 2k<n \footnotemark
\end{cases}
\end{equation}
and
\begin{equation}\label{eq:1b}
\mathcal{G}(n,\overline{[k]})=\begin{cases}
0 & 0\leq n<k\\
n-k & k\leq n\leq3k\\
2k+\lceil\frac{n-3k}{k+1}\rceil & 3k<n.
\end{cases}
\end{equation}
\end{thm}
\footnotetext{$a\mod b$ denotes the smallest non-negative integer that is congruent to $a \mod b$ rather than the entire residue class. }
\begin{proof}
We will approach this proof by alternating between proving base cases 
for \eqref{eq:1a} and \eqref{eq:1b} and then finally using induction to prove the general
case. 
\begin{casenv}
\item $\mathcal{G}(n,\overline{[k]})=0$ for $0\leq n<k$.

The set of legal moves is $\overline{[k]}\cap[n]=\varnothing$. Since there are
no legal moves, we have no options. Thus $\mathcal{G}(n,\overline{[k]})=0$
for $0\leq n<k$. 
\item $\mathcal{G}(n,[k])=n$ for $0\leq n<k$.

The set of legal moves is $[k]\cap[n]=[n]$. Therefore
\begin{equation}\label{eq:3}
(n,[k])=\{(0,[k]),(0,\overline{[k]}),\ldots,(n-1,[k]),(n-1,\overline{[k]})\}.
\end{equation}
From Case (1) we know that $\mathcal{G}(i,\overline{[k]})=0$ for $0\leq i<k$.
Substituting this result into \eqref{eq:3} we get
\begin{equation}
(n,[k])=\{(0,[k]),*0,\ldots,(n-1,[k]),*0\}.
\end{equation}
We proceed with using induction.

\textbf{Base Case:} Since $(0,[k])$ has no options, $\mathcal{G}(0,[k])=0$. 

\textbf{Inductive Step: }Assume that for some $n'<k$ that $\mathcal{G}(n,[k])=n$
for all $n<n'$.
Consider $\mathcal{G}(n'+1,[k])$:
\begin{align}
(n'+1,[k]) &= \{(0,[k]),(0,\overline{[k]}),\ldots,(n',[k]),(n',\overline{[k]})\}
\\
&= \{*0,*0,*1,*0,\ldots,*n',*0\}.
\end{align} 
Thus $\mathcal{G}(n'+1,[k])=n'+1$ by the mex rule.
By induction, $\mathcal{G}(n,[k])=n$. 
\item $\mathcal{G}(n,\overline{[k]})=n-k$ for $k\leq n<2k$.

The set of legal moves is $\overline{[k]}\cap[n]=\{k+1,k+2,\ldots, n\}$.
Therefore,
\begin{equation}\label{eq:6}
(n,[k])=\{(0,[k]),(0,\overline{[k]}),\ldots,(n-k-1,[k]),(n-k-1,\overline{[k]})\}.
\end{equation}
Since $n<2k$, we conclude that $n-k-1<k-1$. Therefore, we know how
to evaluate all the options of $(n,\overline{[k]})$. Substituting those
values into \eqref{eq:6} we get
\begin{equation}
(n,[k])=\{*0,*0,*1,*0,\ldots,*(n-k-1),*0\}.
\end{equation}
Using the mex rule we can evaluate this as $\mathcal{G}(n,[k])=n-k$.
\item $\mathcal{G}(n,[k])=n$ for $k\leq n<2k$.

The set of legal moves is $[k]\cap[n]=[k]$. Therefore
\begin{equation}
(n,[k])=\{(n-k,[k]),(n-k,\overline{[k]}),\ldots,(n-1,[k]),(n-1,\overline{[k]})\}.
\end{equation}
We will prove that $\mathcal{G}(n,[k])=n$ using induction. 

\textbf{Base Case: }We know from Case (2) that $\mathcal{G}(n,[k])=n$
for $0\leq n<k$. This serves as our base case.

\textbf{Inductive Step: }Now assume for some for some $n'<2k$
that it is true that all $\mathcal{G}(n,[k])=n$ for all $n\leq n'$. Consider
$(n'+1,[k])$. We have
\begin{equation}
(n'+1,[k])=\{(n'+1-k,[k]),(n'+1-k,\overline{[k]}),\ldots,(n',[k]),(n',\overline{[k]})\}.
\end{equation}
To evaluate options of the form $(n,[k])$ we use our inductive assumption. The options are
\begin{equation}\label{eq:10}
(n'+1,[k])=\{*(n'+1-k),(n'+1-k,\overline{[k]}),\ldots,*n',(n',\overline{[k]})\}.
\end{equation} 
To evaluate options of the form $\mathcal{G}(i,\overline{[k]})$ we use
our result from Cases (1) and (3). Since $n'+1<2k$, $n'+1-k<k$.
We apply the result from Case (1) to find that $(i,\overline{[k]})=*0$
where $0\leq i<k$. For $k<n<n'+1$, $\mathcal{G}(n,\overline{[k]})=n-k$, which
implies that the options are $*0,*1,\ldots,*(n'+1-k)$. Substituting these results
into \eqref{eq:10} we get
\begin{equation}
(n'+1,[k])=\{*(n'+1-k),*0,\ldots,*n',*(n'-k)\}.
\end{equation}
Using the mex rule here, we determine that $\mathcal{G}(n'+1,[k])=n'+1$. By
 induction $\mathcal{G}(n,[k])=n$ for $k\leq n\leq2k$. 
\item $\mathcal{G}(n,\overline{[k]})=n-k$ for $2k\leq n \leq 3k$.

The set of legal moves is $\overline{[k]}\cap[n]=\{k+1,k+2,\ldots, n\}$.
Therefore the set of options are given by
\begin{equation}\label{eq:16}
(n,\overline{[k]})=\{(0,[k]),(0,\overline{[k]}),\ldots,(n-k-1,[k]),(n-k-1,\overline{[k]})\}.
\end{equation}
Since $n\leq3k$, we know that $n-k-1\leq2k-1$. Therefore, we know
how to evaluate the options of $(n,\overline{[k]})$. Substituting those
values into \eqref{eq:16} we get:
\begin{equation}
(n,\overline{[k]})=\{*0,*0,*1,*0,\ldots,*(n-k-1),*(n-2k-1)\}
\end{equation}
We use the mex rule to find that $\mathcal{G}(n,\overline{[k]})=n-k$.
\item $\mathcal{G}(n,[k])=n+1\mod(k+1)$ for $2k<n$.

The set of legal moves is $[k]\cap[n]=[k]$. Therefore the set of options is 
\begin{equation}
(n,[k])=\{(n-k,[k]),(n-k,\overline{[k]}),\ldots,(n-1,[k]),(n-1,\overline{[k]})\}.
\end{equation}
Now we use induction.

\textbf{Base Cases:} Our base cases will be $(2k+1,[k]),\ldots,(3k+1,[k])$.
The game $(2k+1,[k])$ can be written in terms of its options as 
\begin{equation}
(2k+1,[k])=\{(k+1,[k]),(k+1,\overline{[k]}),\ldots,(2k,[k]),(2k,\overline{[k]})\}.
\end{equation}
Substituting values from our previous casework, we may rewrite the above as 
\begin{equation}
(2k+1,[k])=\{*(k+1),*(1),\ldots,*(2k),*(k)\}
\end{equation}
We use the mex rule to get $\mathcal{G}(2k+1,[k])=0=(2k+2)\mod(k+1)$.
The game $(2k+2,[k])$ can be written in terms of its options as
\begin{equation}
(2k+2,[k])=\{*(k+2),*(2),\ldots,*0,*(k+1)\}.
\end{equation}
We use the mex rule to get $\mathcal{G}(2k+2,[k])=1=(2k+3)\mod(k+1)$.
We use the same process to evaluate our other base cases. \\
\textbf{Inductive Step: } For some $n'>3k+1$
assume that $\mathcal{G}(n,[k])=(n+1)\mod(k+1)$ for all $2k<n\leq n'$.
We evaluate $\mathcal{G}(n'+1,[k])$ as
\begin{equation}\label{eq:21}
(n'+1,[k])=\{(n'-k+1,[k]),(n'-k+1,\overline{[k]}),\ldots,(n',[k]),(n',\overline{[k]})\}.
\end{equation}

Because of the mex rule and Lemma \ref{thm:1}, we may disregard all the
$(n,\overline{[k]})$ because $\mathcal{G}(n,\overline{[k]})\geq2k-1$, while
$*(k+1)$ is missing from the options list (i.e. $\mathcal{G}(n'+1,[k])\leq k+1$).
After removing such options, we can rewrite \eqref{eq:21} as
\begin{equation}
(n'+1,[k])=\{(n'-k+1,[k]),\ldots,(n',[k])\}.
\end{equation}
We use the the mex rule here to see that $\mathcal{G}(n'+1,[k])=\mathcal{G}(n'-k,[k])=n'-k+1\mod(k+1)=n'+2\mod(k+1)$.
By induction, $\mathcal{G}(n,[k])=n+1\mod(k+1)$
for $2k<n$.
\item $\mathcal{G}(n,\overline{[k]})=2k+\lceil\frac{n-3k}{k+1}\rceil$.

The set of legal moves is $\overline{[k]}\cap[n]=\{k+1,k+2,\ldots, n\}$.
Therefore the options are 
\begin{equation}
(n,[k])=\{(0,[k]),(0,\overline{[k]}),\ldots,(n-k-1,[k]),(n-k-1,\overline{[k]})\}.
\end{equation}

\textbf{Base Cases:} The base cases are $(3k+1,\overline{[k]}),\ldots,(4k+1,\overline{[k]})$.We rewrite $(3k+1,\overline{[k]})$ in terms of its options as
\begin{equation}
(3k+1,\overline{[k]})=\{(0,[k]),(0,\overline{[k]}),\ldots,(2k,[k]),(2k,\overline{[k]})\}.
\end{equation}

From our previous casework and the mex rule, we conclude that
$\mathcal{G}(3k+1,\overline{[k]})=2k+1=2k+\lceil\frac{3k+1-3k}{k+1}\rceil$.
The other base cases have very similar options and also evaluate to
$2k+1$. 

\textbf{Inductive Step: }Assume for some $n'>4k+1$ that \eqref{eq:1b} holds
true for $n'-k$ to $n'$ (inclusive). 

From our previous casework, we notice that the options of the form
$(n,[k])$ take on every values in $\{*0,*1,\ldots,*2k\}$and no other
values. Consider $(m,\overline{[k]})$, an option of $(n,\overline{[k]})$. Because
stones need to be taken from $n$ to reach $m$, $m<n$. We know that $\mathcal{G}(n,[k])$
is monotonic from Lemma \ref{thm:1}. However, if $(m,\overline{[k]})$ is an option
of $(n,\overline{[k]})$, then $\mathcal{G}(n,\overline{[k]})\neq \mathcal{G}(m,\overline{[k]})$
by the mex rule. Since $\mathcal{G}(n,\overline{[k]})\geq \mathcal{G}(m,\overline{[k]})$
and $\mathcal{G}(n,\overline{[k]})\neq\mathcal{G}(m,\overline{[k]})$
(by the mex rule), we know that $\mathcal{G}(n,\overline{[k]})>\mathcal{G}(m,\overline{[k]})$.
Therefore, the value of $(n,\overline{[k]})$ is greater than the value
of any of its options. However, by the mex rule, it can only
be 1 greater than any of its options. Using this we will prove the
last part of the formula. We notice that $(n'-k,\overline{[k]})$ is the largest option
of $(n'+1,\overline{[k]})$. Therefore, $\mathcal{G}(n'+1,\overline{[k]})=\mathcal{G}(n'-k,\overline{[k]})+1=2k+\lceil\frac{n'-k-3k}{k+1}\rceil+1=2k+\lceil\frac{n'+1-3k}{k+1}\rceil$.
By induction \eqref{eq:1b} is true. 
\end{casenv}
\end{proof}

\section{Arithmetic Progressions} \label{sec:arith}

Now we will proceed with investigating Grundy values for games where
$S=\{b+ic:\ 0\leq i\leq i_{\max},\ b,c,i,i_{\max}\in\mathbb{N},\frac{c+2}{2}\leq b<c\}$. 
Here we extend the results in \cite{key-2} where subtraction sets
of the form $\{b+ic:\ 0\leq i,\ b,c,i\in\mathbb{N},\frac{c+2}{2}\leq b<c\}$
were considered by considering the behavior of Grundy values when the infinite arithmetic progression is chopped off to form a finite arithmetic progression. For any given subtraction set $S = \{b+ic:\ 0\leq i\leq i_{\max},\ b,c,i,i_{\max}\in\mathbb{N},\frac{c+2}{2}\leq b<c\}$
we call set $S'=\{b+ic:\ 0\leq i,\ b,c,i\in\mathbb{N}\}$. This helps
us use the results of \cite{key-2} in proving our results.

We observe that the sequence of values for $\mathcal{G}(n,S)$
is eventually periodic with period $p=2b+i_{\max}c$.
A list of Grundy values of $(n,S)$ and $(n,\overline{S})$ where $S = \{8, 21, 34, 47\}$ is given as an example in Figure~\ref{fig:grundytable}.

\begin{figure}
\includegraphics[clip,height=8in]{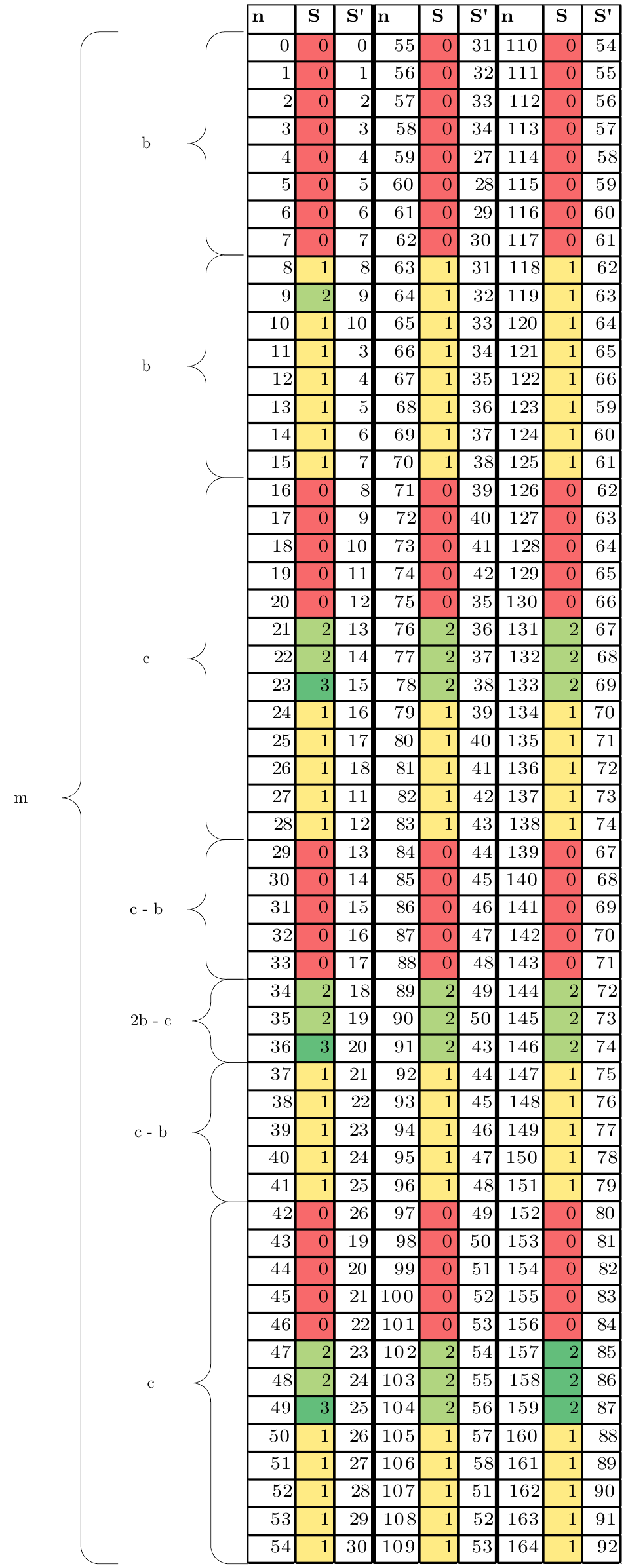}

\caption{Labeled Grundy value table for $S=\{8,21,34,47\}$} 
\label{fig:grundytable}
\end{figure}

\begin{lem} \label{thm:3.1}
If $n<p$, $\mathcal{G}(n,S)=\mathcal{G}(n,S')$ and $\mathcal{G}(n,\overline{S})=\mathcal{G}(n,\overline{S'})$. 
\end{lem}

\begin{proof}
Since $p=2b+i_{\max}c<b+(i_{\max}+1)c$, the extra moves
in $S'$ do not affect the game, since the next player may not remove more than
$n$ stones from $(n,S)$ anyway. So the options of $(n,S)$ and
$(n,S')$ are identical when $n<p$. Therefore, $\mathcal{G}(n,S)=\mathcal{G}(n,S')$
and $\mathcal{G}(n,\overline{S})=\mathcal{G}(n,\overline{S'})$.
\end{proof}
\begin{lem}\label{thm:3.2}
If $n\geq2$, then $\mathcal{G}(n,S)\geq2$.
\end{lem}

\begin{proof}
This proof is virtually identical to the proof provided for a similar
result in \cite{key-2}. If $n-1\in\overline{S}$, then $(n,\overline{S})$
has the options $(1,S)$ and $(1,\overline{S})$. $\mathcal{G}(1,S)=0$
and $\mathcal{G}(1,S)=1$, therefore $\mathcal{G}(n,S)\geq2$. If
$n-1\in S$, $n\in\overline{S}$ and $n-b\in\overline{S}$. Since $(0,S)=*0$
and $(b,S)=*1$ are options of $(n,\overline{S})$, $\mathcal{G}(n,S)\geq2$.
\end{proof}
\begin{lem}\label{thm:3.3}
If $0\leq j\leq b-1$, $\mathcal{G}(p+j,S)=0$.
\end{lem}

\begin{proof}
The options of $(p+j,S)$ are of the form $(b+kc+j,S)$ and $(b+kc+j,\overline{S})$
where $0<k<i_{\max}$. Since in these options the number of stones
is less than $p$ we may evaluate them by using Lemma \ref{thm:3.1}. By using Lemma \ref{thm:3.1}, we see that they are all greater than $0$. Therefore,
$\mathcal{G}(p+j,S)=0$ by the mex rule.
\end{proof}

\begin{lem}\label{thm:3.4}
If $b\leq j\leq2b-1$, $\mathcal{G}(p+j,S)=1$.
\end{lem}

\begin{proof}
First, we consider games where $b\leq j\leq c$. $(2b+i_{\max}c+j,S)$
has options of the form $(b+kc+j,S)$ and $(b+kc+j,\overline{S})$. Since
$b\leq j\leq c$, these can be rewritten as $(2b+kc+x,S)$ and $(2b+kc+x,\overline{S})$
for some $0\leq x\leq c-b$. From the results in \cite{key-2} and
our result in Lemma \ref{thm:3.3}, these options all evaluate to 0. If $c\leq j\leq2b-1$,
we may rewrite $(b+kc+j,S)$ and $(b+kc+j,\overline{S})$ as $(b+(k+1)c+x,S)$
and $(b+(k+1)c+x,\overline{S})$ for $0\leq x\leq2b-c-1$. From Lemma \ref{thm:3.2}
we know $\mathcal{G}(b+(k+1)c+x,\overline{S})>2$. $\mathcal{G}(b+(k+1)c+x,S)>1$
from the results in \cite{key-2}. Therefore, $\mathcal{G}(p+j,S)=1$.
\end{proof}
\begin{lem}\label{thm:3.5}
If $0\leq i'\leq i_{\max}$ and $0\leq j\leq c-b$, then $\mathcal{G}(p+2b+i'c+j,S)=0$.
\end{lem}

\begin{proof}
Consider the games of form $(p+2b+i'c+j,S)$ where $0\leq j\leq c-b$.
These games have options of the form $(3b+(i_{\max}+i'-k)c+j,S)$ and
$(3b+(i_{\max}+i'-k)c+j,\overline{S})$ where $0\leq k\leq i_{\max}$. As we will see, the
options of the latter kind can be ignored since they are all greater
than or equal to 2. We consider the options described before. If $3b+(i_{\max}+i'-k)c+j\leq p$
then that option evaluates to $1$ from the results in \cite{key-2}.
If $3b+(i_{\max}+i'-k)c+j>p$, then we consider the option's options.
The option has an option of the form $(2b+i'c+j,S)=*0$ by results
described in \cite{key-2}. Therefore, the option itself cannot b.
$*0$. Additionally, we note that none of the options with $\overline{S}$ as the subtraction set have values equal to $0$ or $1$ by Lemma \ref{thm:3.2}.
Therefore, we use the mex rule to evaluate the game as $\mathcal{G}(p+2b+i'c+j,S)=0$.
\end{proof}
\begin{lem}\label{thm:3.6}
If $0\leq i'\leq i_{\max}$ and $0\leq j\leq c-b$, then $\mathcal{G}(p+i'c+3b+j,S)=1$.
\end{lem}

\begin{proof}
The options of $(p+i'c+3b+j,S)$ are of the form $(4b+kc+j,S)$. If
$4b+kc+j>p$ then we know from Lemma \ref{thm:3.5} that $(4b+kc+j,S)=*0$. Else,
we can rewrite $(4b+kc+j,S)$ as $(2b+(k+1)c+(2b-c)+j,S)$. From the
results in \cite{key-2}, we know that these options are either 0
or greater than 1 (they all have options that are $*1$). Therefore,
$\mathcal{G}(p+i'c+3b+j,S)=1$.
\end{proof}
\begin{lem}\label{thm:3.7}
If $0\leq i'\leq i_{\max}$ and $0\leq j\leq2b-c$, then $\mathcal{G}(p+i'c+3b+j,S)>1$.
\end{lem}

\begin{proof}
The options of $(p+i'c+3b+j,S)$ has options of the form $(p+2b-c+j,S)$.
This option evaluates to 0 by Lemma \ref{thm:3.3}. Other options of the form
$(p+2b+jc+j,S)$ evaluate to 1. Therefore, using the mex rule, all
of these games have Grundy values greater than 1. 
\end{proof}
\begin{lem}\label{thm:3.8}
If $n\geq p$ then $\mathcal{G}(n,\overline{S}) > 2i_{\max}$.
\end{lem}

\begin{proof}
Since $n\geq p$, $(n,\overline{S})$ has all of $(0,\overline{S}),(0,\overline{S}),\ldots,(p-1,\overline{S})$
as options. From \cite{key-2} we know that the Grundy values of these
options contain every value from 
$0$ to $\max(\{\mathcal{G}(0,\overline{S}),\allowbreak \mathcal{G}(0, \overline{S}), \ldots, G(p-1,\overline{S})\})$.
For $n<p$ we know that if $n=b+3+2br+i$ then $\mathcal{G}(n,\overline{S})=3+rb+i$.
From this we know: $2\times\mathcal{G}(n,\overline{S})=2br+6+2i$. Subtracting
$n$ from both sides, we have $2\times\mathcal{G}(n,\overline{S})-n=-b+3+i$.
By definition $0\leq i<2b$, so $-b+3\leq2\mathcal{G}(n,\overline{S)}-n\leq b+3$.
That is, $2\mathcal{G}(n,\overline{S})\geq n-b+3$ or $\mathcal{G}(n,\overline{S})\geq\frac{{n-b+3}}{2}$.
Since the max $n$ possible is $b+i_{\max}c$, we substitute that in
here and see that $\mathcal{G}(n,\overline{S})>\frac{i_{\max}c+3}{2}$. Since we assumed
$b\geq5\text{ and }c>b$, $c>5$. This implies $\frac{i_{\max}c+3}{2}>2i_{\max}$.
Therefore, $\mathcal{G}(n,\overline{S}) > 2i_{\max}$.
\end{proof}
\begin{thm}\label{thm:3.9}
$\mathcal{G}(n,S)$ is eventually periodic with period $p$. 
\end{thm}

\begin{proof}
First, we notice that for $n>2p$, there are at most $2i_{\max}$ options
for the game $(n,S)$, therefore the maximum value of $(n,S)\leq2i_{\max}$.
All of these options have more than $p$ stones, therefore, the options
with set $\overline{S}$ are all greater than $2i_{\max}$ and therefore
do not affect the value of $(n,S)$. We will prove that given the
following statements for some $l$, they will also hold for $l+1$:
\begin{itemize}
\item if $0\leq j\leq b-1$, $\mathcal{G}(lp+j,S)=0$;
\item if $b\leq j\leq2b-1$, $\mathcal{G}(lp+j,S)=1$;
\item if $0\leq i'\leq i_{\max}$ and $0\leq j\leq c-b$, $\mathcal{G}(lp+2b+i'c+j,S)=0$;
\item if $0\leq i'\leq i_{\max}$ and $0\leq j\leq c-b$, then $\mathcal{G}(lp+i'c+3b+j,S)=1$;
\item i $0\leq i'\leq i_{\max}$ and $0\leq j\leq2b-c$, then $\mathcal{G}(lp+i'c+3b+j,S)>1$.
\end{itemize}
We first note that all of these conditions are met when $l=1$ because of Lemmas \ref{thm:3.1}--\ref{thm:3.8} which serve as the base case. Proving all of the above conditions is sufficient prove that
$\mathcal{G}(n,S)$ is periodic with period $p$ since the only group
of numbers which is not given a constant value ($0\leq i'\leq i_{\max}$
and $0\leq j\leq2b-c$, $\mathcal{G}(lp+i'c+3b+j,S)$) does not have
options of the form $\mathcal{G}((l-1)p+i'c+3b+j,S)$). All its options
from the previous block are given constant values by the conditions listed above, therefore they must be the same in each block. Now we proceed to prove that if the above conditions hold for some $l$, they hold for $l+1$. 

First we consider games of the form $((l+1)p+j,S)$ where $0\leq j\leq b-1$. The options of $((l+1)p+j,S)$ are of the form $(lp+b+kc+j,S)$ and $(lp+b+kc+j,\overline{S})$
where $0<k<i_{\max}$. These options are all greater than 0 by our inductive hypothesis. Therefore, $\mathcal{G}((l+1)p+j,S)=0$
by the mex rule.

Next we consider games of the form $\mathcal{G}((l+1)p+j,S)=1$ where $b\leq j\leq2b-1$.
First, we consider games where $b\leq j\leq c$. The game $((l+1)p+2b+i_{\max}c+j,S)$
has options of the form $(lp+b+kc+j,S)$ and $(lp+b+kc+j,\overline{S})$.
Since $b\leq j\leq c$, these can be rewritten as $(lp+2b+kc+x,S)$
and $(lp+2b+kc+x,\overline{S})$ for some $0\leq x\leq c-b$. From
the conditions from the inductive hypothesis listed before, and our previous observation, these
options all evaluate to 0. If $c\leq j\leq2b-1$, we may rewrite $(lp+b+kc+j,S)$
and $(lp+b+kc+j,\overline{S})$ as $(lp+b+(k+1)c+x,S)$ and $(lp+b+(k+1)c+x,\overline{S})$
for $0\leq x\leq2b-c-1$. Therefore, $\mathcal{G}((l-1)p+b+(k+1)c+x,S)>1$ because of the conditions from the inductive hypothesis. Therefore, $((l+1)p+j,S)=1$.

Next we consider the games of form $((l+1)p+2b+i'c+j,S)$ where $0\leq j\leq c-b$.
These games have options of the form $lp+3b+(i_{\max}+i'-k)c+j,S)$
and $lp+3b+(i_{\max}+i'-k)c+j,\overline{S})$ where $0\leq k\leq i_{\max}$.
The options of the latter kind can be ignored since they are all greater
than or equal to $2 i_{\max}$. We consider the options described before. If
$lp+3b+(i_{\max}+i'-k)c+j\leq(l+1)p$ then that option evaluates to
$1$ from the results in \cite{key-2}. If $lp+3b+(i_{\max}+i'-k)c+j>(l+1)p$,
then we consider the option's options. The option has an option of
the form $(lp+2b+i'c+j,S)=*0$ by our assumptions. Therefore, the
option itself cannot be $*0$. Using the mex rule, we can see now
that $\mathcal{G}((l+1)p+2b+i'c+j,S)=0$.

The options of $((l+1)p+i'c+3b+j,S)$ are of the form $(lp+4b+kc+j,S)$.
If $lp+4b+kc+j>(l+1)p$ then we know from Lemma \ref{thm:3.5} that $(4b+kc+j+lp,S)=*0$.
Else, we can rewrite $(4b+kc+j+lp,S)$ as $(2b+(k+1)c+(2b-c)+j+lp,S)$.
From the assumptions, we know that these options are either 0 or greater
than 1 (they all have options that are $*1$). Therefore, n $\mathcal{G}((l+1)p+i'c+3b+j,S)=1$.

The options of $((l+1)p+i'c+3b+j,S)$ has options of the form $((l+1)p+2b-c+j,S)$.
This option evaluates to 0 by our assumptions. Other options of the
form $((l+1)p+2b+jc+j,S)$ evaluate to 1. Therefore, using the mex
rule, all of these games have Grundy values greater than 1. 

All of the conditions specified in the inductive assumption are met. Therefore, by induction, the $\mathcal{G}(n,S)$ is periodic with period $p$.
\end{proof}

\bibliographystyle{alpha}
\bibliography{muller}

\end{document}